\documentclass[leqn,10pt]{amsart}
\usepackage[utf8]{inputenc}
\usepackage{a4}
\usepackage[safeinputenc=true,style=alphabetic,firstinits,maxnames=5,useprefix=true,backend=bibtex]{biblatex}

\bibliography{birkhoff}
\usepackage{amsmath,amssymb,amscd,bbm}
\usepackage{amsthm}
\usepackage{enumerate}
\usepackage[mathscr]{eucal}
\usepackage[all]{xy}
\newtheorem{thm}{Theorem}[section]
\newtheorem*{thmn}{Theorem}
\newtheorem{lemma}[thm]{Lemma}
\newtheorem{prop}[thm]{Proposition}
\newtheorem{cor}[thm]{Corollary}

\theoremstyle{definition}

\newtheorem{exa}[thm]{Example}

\theoremstyle{remark}

\newtheorem{rem}[thm]{Remark}

\numberwithin{equation}{section}

\newcommand{\vanish}[1]{\relax}       % comments out inclosed text
\def\qedsymbol{\hbox to 1ex{\llap{\rule{0.25pt}{1ex}}\rlap{\rule{1ex}{0.25pt}}\lower0.25pt\rlap{\raise1ex\rlap{\rule{1ex}{0.25pt}}}\hskip1ex\llap{\rule{0.25pt}{1ex}}}}
\def\rlqed{\rlap{\rule{\hsize}{0pt}\kern-1ex\kern-1em\qed}} %% \qed rechts

\newcounter{aufzi}
\newenvironment{aufzi}{\begin{list}{ {\upshape(\alph{aufzi})}}{
        \usecounter{aufzi}
        \topsep1ex
        \parsep0cm
        \itemsep0.8ex
        \leftmargin1cm
        \labelwidth0.5cm
        \labelsep0.3cm
}}
{\end{list}}

\newcounter{aufzii}

\newcounter{aufziii}
\newenvironment{aufziii}{\begin{list}{ {\upshape\arabic{aufziii})}}{
        \usecounter{aufziii}
        \topsep1ex
        \parsep0cm
        \itemsep0.8ex
        \leftmargin1cm
        \labelwidth0.5cm
        \labelsep0.3cm
}}
{\end{list}}



\newcommand{\bbE}{\mathbb{E}}

\newcommand{\calB}{\mathcal{B}}

\def\bfX{\mathbf{X}}

\def\ue{\mathrm{e}}

\def\uw{\mathrm{w}}

\renewcommand{\ue}{\mathrm{e}}    %  for Euler's number
\newcommand{\R}{\mathbb{R}}     % real numbers etc
\newcommand{\N}{\mathbb{N}}
\newcommand{\Z}{\mathbb{Z}}
\newcommand{\Q}{\mathbb{Q}}

\newcommand{\sdif}{\triangle}      % symmetric difference

\newcommand{\Bigcap}[2][\relax]{%
 \ifx#1\relax \bigcap_{#2}
 \else \bigcap^{#1}_{#2}
 \fi}
\newcommand{\Bigcup}[2][\relax]{%
 \ifx#1\relax \bigcup_{#2}
 \else \bigcup^{#1}_{#2}
 \fi}

\DeclareMathOperator{\Aut}{Aut}

\def\fact#1#2{#1/#2}
\def\tfact#1#2{#1/#2}

\def\fact#1#2{{\raise0.2em\hbox{$#1$}\kern-0.2em/\kern-0.1em\lower0.2em\hbox{$#2$}}}
\def\tfact#1#2{{\raise0.1em\hbox{\small$#1$}\kern-0.1em/\kern-0.1em\lower0.1em\hbox{\small$#2$}}}

\newcommand{\norm}[2][\relax]{%                             % norm
   \ifx#1\relax \ensuremath{\left\Vert#2\right\Vert}
   \else \ensuremath{\left\Vert#2\right\Vert_{#1}}
   \fi}

\newcommand{\Bnorm}[2][\relax]{%        % Big norm, but usually no too big
   \ifx#1\relax \ensuremath{\Bigl\Vert#2\Bigr\Vert}
   \else \ensuremath{\Bigl\Vert#2\Bigr\Vert_{#1}}
   \fi}
\makeatletter
\newcommand{\tdprod}[2]{\ensuremath{%
  \setbox0=\hbox{\ensuremath{\langle#1,#2 \rangle}}
  \dimen@\ht0
  \advance\dimen@ by \dp0 (#1\rule[-\dp0]{0pt}{\dimen@}\,|#2\hspace{1pt})}}
\newcommand{\dprod}[2]{\ensuremath{%
  \setbox0=\hbox{\ensuremath{\left\langle#1,#2\right\rangle}}
  \dimen@\ht0
  \advance\dimen@ by \dp0 \left\langle\left.#1\rule[-\dp0]{0pt}{\dimen@}\,\right|#2\hspace{1pt}\right\rangle}}

\newcommand{\bdprod}[2]{\ensuremath{%
  \setbox0=\hbox{\ensuremath{\bigl\langle#1,#2\bigr\rangle}}
  \dimen@\ht0
  \advance\dimen@ by \dp0 \bigl\langle#1\bigl|\rule[-\dp0]{0pt}{\dimen@}\bigr.#2\hspace{1pt}\bigr\rangle}}
\newcommand{\Bdprod}[2]{\ensuremath{%
  \setbox0=\hbox{\ensuremath{\Bigl\langle#1,#2\Bigr\rangle}}
  \dimen@\ht0
  \advance\dimen@ by \dp0 \Bigl\langle#1\Bigl|\rule[-\dp0]{0pt}{\dimen@}\Bigr.#2\hspace{1pt}\Bigr\rangle}}
\newcommand{\tsprod}[2]{\ensuremath{%
  \setbox0=\hbox{\ensuremath{(#1,#2)}}
  \dimen@\ht0
  \advance\dimen@ by \dp0 (#1\rule[-\dp0]{0pt}{\dimen@}\,|#2\hspace{1pt})}}
\newcommand{\sprod}[2]{\ensuremath{%
  \setbox0=\hbox{\ensuremath{\left(#1,#2\right)}}
  \dimen@\ht0
  \advance\dimen@ by \dp0 \left(\left.#1\rule[-\dp0]{0pt}{\dimen@}\,\right|#2\hspace{1pt}\right)}}
\newcommand{\bsprod}[2]{\ensuremath{%
  \setbox0=\hbox{\ensuremath{\bigl(#1,#2\bigr)}}
  \dimen@\ht0
  \advance\dimen@ by \dp0 \bigl(#1\bigl|\rule[-\dp0]{0pt}{\dimen@}\bigr.#2\hspace{1pt}\bigr)}}
\newcommand{\Bsprod}[2]{\ensuremath{%
  \setbox0=\hbox{\ensuremath{\Bigl(#1,#2\Bigr)}}
  \dimen@\ht0
  \advance\dimen@ by \dp0 \Bigl(#1\Bigl|\rule[-\dp0]{0pt}{\dimen@}\Bigr.#2\hspace{1pt}\Bigr)}}
\makeatother

\def\bell#1{\ell^{\mathrm{#1}}}   % l^p spaces

\newcommand{\Ell}[2][\relax]{%     % Ell-p spaces
   \ifx#1\relax \mathrm{L}^{\mathrm{#2}}
   \else \mathrm{L}^{\mathrm{#2}}_{\mathrm{#1}}
   \fi}
\renewcommand{\Ell}[2][\relax]{%    ??????????????????????????????????????????
   \ifx#1\relax \mathrm{L}^{\!#2}
   \else \mathrm{L}^{\!#2}_{\mathrm{#1}}
   \fi}

\newcommand{\Wee}[2][\relax]{%    %   Sobolev spaces
   \ifx#1\relax \mathrm{W}^{\mathrm{#2}}
   \else \mathrm{W}^{\mathrm{#2}}_{\mathrm{#1}}
   \fi}
\newcommand{\Har}[2][\relax]{%       %Hardy spaces
   \ifx#1\relax \mathsf{H}^{\mathsf{#2}}
   \else   \mathsf{H}^{\mathsf{#2}}_{\mathrm{#1}}
   \fi}
\def\prX{\mathrm X}    % probability space
\def\rlqed{\rlap{\rule{\hsize}{0pt}\kern-1ex\kern-1em\qed}}

\makeatletter

\def\maketag@@@@@#1{\llap{\hbox to\hsize{\m@th\normalfont#1}}%
\gdef\tagform@##1{\maketag@@@{(\ignorespaces##1\unskip\@@italiccorr)}}}

\def\eqtext#1{\gdef\tagform@##1{\maketag@@@@@{\ignorespaces##1\unskip\@@italiccorr\hfill}}\tag{#1}}%
\def\reqtext#1{\gdef\tagform@##1{\maketag@@@@@{\hfill\ignorespaces##1\unskip\@@italiccorr}}\tag{#1}}%
\def\leqtext#1{\gdef\tagform@##1{\maketag@@@@@{\ignorespaces##1\unskip\@@italiccorr}}\tag{#1}}%
\makeatother

\newcommand{\cntm}[1]{\left|#1\right|}
\newcommand{\Rp}{\R_{\geq 0}}

\newcommand{\Zp}{\Z_{\geq 0}}
\newcommand{\bl}{\mathrm B}

\newcommand{\cin}[1]{\mathrm{I}(#1)}

\newcommand{\avg}[1]{\bbE_{#1}}
\newcommand{\indi}[1]{\mathbf{1}_{#1}}

\date{\today}

\begin{document}

\title[Effective Birkhoff's theorem]{On effective Birkhoff's ergodic theorem\\
for computable actions of amenable groups.}

\author[N. Moriakov]{Nikita Moriakov}

\email{n.moriakov@tudelft.nl}

\date{\today}

\begin{abstract}
We introduce computable actions of computable groups and
prove the following versions of effective Birkhoff's ergodic
theorem. Let $\Gamma$ be a computable amenable group, then there
always exists a canonically computable tempered two-sided F{\o}lner sequence $(F_n)_{n \geq
  1}$ in $\Gamma$. For a computable, measure-preserving, ergodic action of $\Gamma$
on a Cantor space $\{ 0,1\}^{\mathbb N}$ endowed with a computable probability
measure $\mu$, it is shown that for every bounded
lower semicomputable function $f$ on $\{0,1\}^{\mathbb N}$ and for every Martin-L\"{o}f random $\omega \in \{0,1\}^{\mathbb N}$
the equality
\[
\lim\limits_{n \to \infty} \frac{1}{|F_n|} \sum\limits_{g \in F_n} f(g
\cdot \omega) = \int\limits f d \mu
\]
holds, where the averages are taken with respect to a canonically
computable tempered two-sided F{\o}lner sequence $(F_n)_{n \geq
  1}$. We also prove the same identity for \emph{all} lower semicomputable
$f$'s in the special case when $\Gamma$ is a computable group of polynomial growth and
$F_n:=\mathrm{B}(n)$ is the F{\o}lner sequence of balls around the
neutral element of $\Gamma$.
\end{abstract}

\maketitle

\section{Introduction}
\label{s.intro}

A classical ergodic theorem of Birkhoff asserts that, if $\varphi: X \to X$ is an ergodic measure-preserving transformation on a probability space $(X,\mu)$, then for every $f \in \Ell{1}(X)$ we have
\begin{equation}
\label{eq.birkthmi}
\lim\limits_{n \to \infty} \frac 1 n \sum\limits_{i=1}^n f(\varphi^i x) = \int f d \mu
\end{equation}
for $\mu$-a.e. $x \in X$. We refer, e.g., to \cite[Chapter
11]{efhn2015} for the proof. A celebrated result of Lindenstrauss \cite{lindenstrauss2001} gives a generalization of Birkhoff's ergodic theorem for measure-preserving actions of amenable groups and ergodic averages, taken along tempered F{\o}lner sequences.

One may also wonder if the averages in Equation \eqref{eq.birkthmi}
converge for every Martin-L\"{o}f random $x$ and every computable
$f$. An affirmative answer was given by V'yugin in \cite{vyugin1997} for computable
$f$'s. Later, it was proved in \cite{shen2012} that the ergodic averages converge for all lower semi-computable $f$'s. 

In so far, the effective ergodic theorems have only been proved for
actions of $\Z$, and it is a natural question if one can generalize
effective Birkhoff's ergodic theorem for measure-preserving actions of more general
groups (for instance, the groups $\Z^d$, groups of polynomial growth
and so on). However, one must first define \emph{computable actions}
of groups appropriately. In this article we define computable actions
of computable groups in a natural way in Section \ref{ss.cps}, which agrees with the
`classical' definition in the case of $\Z$-actions, and obtain the
following generalizations of the results from \cite{shen2012}. First of
all, we derive a generalization of Ku\v{c}era's theorem in Section \ref{ss.kt}, which is the main technical tool of the article.
\begin{thmn}
Let $\Gamma$ be a computable amenable group and $(\{ 0,1\}^{\N},\mu,\Gamma)$ be a
computable ergodic $\Gamma$-system. Let $U \subset \{ 0,1\}^{\N}$ be an effectively
open subset  such that $\mu(U)<1$. Let 
\[
U^{\ast}:=\bigcap\limits_{g \in \Gamma} g^{-1} U
\]
be the set of all points $\omega \in \{ 0,1\}^{\N}$ whose orbit remains in
$U$. Then $U^{\ast}$ is an effectively null set.
\end{thmn}

Using this generalization of Ku\v{c}era's theorem and the results of
Lindenstrauss, we derive the first main theorem in Section \ref{ss.mt}. To simplify the notation, we denote the averages by
$\avg{g \in F}:=\frac{1}{\cntm{F}}\sum\limits_{g \in F}$.

\begin{thmn}
Let $\Gamma$  be a computable amenable group with a canonically computable tempered
two-sided F{\o}lner sequence $(F_n)_{n \geq 1}$. Suppose that
$(\{ 0,1\}^{\N},\mu,\Gamma)$ is a computable ergodic $\Gamma$-system. For every bounded
lower semicomputable $f$ and for every Martin-L\"{o}f random $\omega \in \{ 0,1\}^{\N}$
the equality
\[
\lim\limits_{n \to \infty} \avg{g \in F_n} f(g
\cdot \omega) = \int\limits f d \mu
\]
holds.
\end{thmn}

In a special case, when $\Gamma$ is a computable group of polynomial
growth, we are able to remove the boundedness assumption on $f$ and prove the following version of effective Birkhoff's ergodic theorem.
\begin{thmn}
Let $\Gamma$  be a computable group of polynomial growth with the F{\o}lner
sequence of balls around $\ue \in \Gamma$ given by
\[
F_n:= \{ g \in \Gamma: \| g \| \leq n\} \quad \text{for } n \geq 1.
\]
Suppose that $(\{ 0,1\}^{\N},\mu,\Gamma)$ is a computable ergodic $\Gamma$-system. For every lower semicomputable $f$ and for every Martin-L\"{o}f random $\omega \in \{ 0,1\}^{\N}$
the equality
\[
\lim\limits_{n \to \infty} \avg{g \in F_n} f(g
\cdot \omega) = \int\limits f d \mu
\]
holds.
\end{thmn}

\section{Preliminaries}
\subsection{Computable Amenable Groups}
\label{ss.cag}
In this section we will remind the reader of the classical notion of
amenability and state some results from ergodic theory of amenable
group actions. We stress that all the groups that we consider are
discrete and countably infinite. 

Let $\Gamma$ be a group with the counting measure $\cntm{\cdot}$. A
sequence of finite subsets $(F_n)_{n \geq 1}$ of $\Gamma$ is called
\begin{aufziii}
\item a \textbf{left F{\o}lner sequence} (resp. \textbf{right F{\o}lner sequence}) if for every  $g \in \Gamma$ one has
\begin{equation*}
    \frac{\cntm{F_n \sdif g F_n}}{\cntm{F_n}} \to 0 \quad \left( \text{resp. } \frac{\cntm{F_n \sdif F_n g}}{\cntm{F_n}} \to 0 \right);
\end{equation*}

\item a \textbf{($C$-)tempered sequence} if there is a constant $C$ such that for every $j$ one has $$\cntm{\bigcup\limits_{i<j} F_i^{-1} F_j} < C \cntm{F_j}.$$
\end{aufziii}

A group is called \textbf{amenable} if it has a left F{\o}lner
sequence. A sequence of finite subsets $(F_n)_{n \geq 1}$ of $\Gamma$
is called a \textbf{two-sided F{\o}lner sequence} if it is a left and
a right F{\o}lner sequence simultaneously.

We refer the reader, e.g., to \cite{shnver2003} for the standard
notions of a computable function and a computable/enumerable set,
which will appear in this article. A sequence of finite subsets
$(F_n)_{n \geq 1}$ of $\N$ is called \textbf{canonically computable}
if there is an algorithm that, given $n$, prints the set $F_n$ and
halts. Formally speaking, for a finite set $A=\{ x_1,x_2,\dots,x_k\}
\subset \N$, we call the number $\cin{A}:=\sum\limits_{i=1}^k 2^{x_i}$
the \textbf{canonical index} of $A$. Hence a sequence $(F_n)_{n \geq
  1}$ of finite subsets of $\N$ is canonically computable if and only
if the (total) function $n \mapsto \cin{F_n}$ is computable.

A group $\Gamma$ with the composition operation $\circ$ is called a
\textbf{computable group} if, as a set, $\Gamma$ is a computable
subset of $\N$ and the total function $\circ: \Gamma \times \Gamma \to
\Gamma$ is computable. It is easy to show that in a computable group
$\Gamma$ the inversion operation $g \mapsto g^{-1}$ is a total computable
function. We refer the reader to \cite{rabin1960} for
more details.

Any discrete amenable group $\Gamma$ admits a two-sided F{\o}lner
sequence. Furthermore, if the group is computable, then there exists a
canonically computable two-sided F{\o}lner sequence. To prove that we will need the following result.
\begin{lemma}
\label{l.owlemma}
Given a discrete amenable group $\Gamma$, for any finite symmetric set $K
\subset \Gamma$ such that $\ue \in \Gamma$ and any $\varepsilon>0$ there
exists a finite subset $F \subset \Gamma$ such that
\begin{equation}
\label{eq.2sidedfs}
\cntm{K F K} - \cntm{F} \leq \varepsilon \cntm{F}.
\end{equation}
\end{lemma}
\noindent We refer the reader to \cite[I.\S 1, Proposition 2]{ow1987} for the
proof. 

\begin{lemma}
\label{l.ex2stfs}
Let $\Gamma$ be a computable amenable group. Then there exists a
canonically computable two-sided F{\o}lner sequence $(F_n)_{n \geq 1}$.
\end{lemma}
\begin{proof}
First of all, observe that given $K \subset \Gamma$,
$\varepsilon>0$ as in Lemma \ref{l.owlemma} and a finite set $F
\subset \Gamma$ satisfying Equation \eqref{eq.2sidedfs}, we
have
\[
\frac{\cntm{g F \setminus F}}{\cntm{F}} \leq \varepsilon
\]
and
\[
\frac{\cntm{F g \setminus F}}{\cntm{F}} \leq \varepsilon
\]
for all $g \in K$. Let $K_n$ be the finite set of the first $n$
elements of the computable group $\Gamma$. Then, for every
$n=1,2,\dots$ we apply Lemma \ref{l.owlemma} to the set $K_n \cup
K_n^{-1} \cup \{ \ue \}$ and
$\varepsilon_n:=1/n$ and find the finite set $F_n$ with the smallest
canonical index $\cin{F_n}$ satisfying Equation
\eqref{eq.2sidedfs}. It is easy to see that $(F_n)_{n \geq 1}$ is
indeed a two-sided F{\o}lner sequence.
\end{proof}

Every F{\o}lner sequence has a tempered F{\o}lner
 subsequence. Furthermore, the construction of a tempered F{\o}lner
 subsequence from a given canonically computable F{\o}lner sequence is
 `algorithmic'. The proof is essentially contained in
 \cite[Proposition 1.4]{lindenstrauss2001}, but we provide it for
 reader's convenience below.
\begin{prop}
\label{pr.tempsseq}
Let $(F_n)_{n \geq 1}$ be a canonically computable F{\o}lner sequence in a computable group $\Gamma$. Then there is a computable function $i \mapsto n_i$ s.t. the subsequence $(F_{n_i})_{i \geq 1}$ is a canonically computable tempered F{\o}lner subsequence.
\end{prop}
\begin{proof}
We define $n_i$ inductively as follows. Let $n_1:=1$. If $n_1,\dots,n_i$ have been determined, we set $\widetilde F_i:= \bigcup\limits_{j \leq i} F_{n_j}$. Take for $n_{i+1}$ the first integer greater than $i+1$ such that
\begin{equation*}
\cntm{F_{n_{i+1}}  \sdif \widetilde F_i^{-1} F_{n_{i+1}}} \leq \frac 1 {\cntm{\widetilde F_i}}
\end{equation*}
The function $i \mapsto n_i$ is total computable. It follows that
\begin{equation*}
\cntm{ \bigcup\limits_{j \leq i} F_{n_j}^{-1} F_{n_{i+1}}} \leq 2 \cntm{ F_{n_{i+1}}},
\end{equation*}
hence the sequence $(F_{n_i})_{i \geq 1}$ is $2$-tempered. Since the F{\o}lner sequence $(F_n)_{n \geq 1}$ is canonically computable and the function $i \mapsto n_i$ is computable, the F{\o}lner sequence $(F_{n_i})_{i \geq 1}$ is canonically computable and tempered.
\end{proof}

\noindent Let us state an immediate corollary.
\begin{cor}
\label{c.cc2sidedfs}
Let $\Gamma$ be a computable amenable group. Then there exists a
canonically computable, tempered two-sided F{\o}lner sequence
$(F_n)_{n \geq 1}$ in $\Gamma$.
\end{cor}

The following result tells us that the $\limsup$ of averages of
bounded functions on an amenable group is translation-invariant.
\begin{lemma}[Limsup invariance]
\label{l.lsinv1}
Let $\Gamma$ be an amenable group with a right F{\o}lner sequence $(F_n)_{n
  \geq 1}$ and $f\in \bell{\infty}(\Gamma, \R)$ be a
bounded function on $\Gamma$. Then
\[
\limsup\limits_{n \to \infty} \avg{g \in F_n} f(g) = \limsup\limits_{n
  \to \infty} \avg{g \in F_n} f(g h).
\]
\end{lemma}
\begin{proof}
A direct computation shows that for all $n \geq 1$
\[
\frac{1}{\cntm{F_n}}\cntm{\sum\limits_{g \in F_n} f(g) -
  \sum\limits_{g \in F_n h} f(g)} \leq \frac{2 \cntm{F_n \setminus F_n
  h} \cdot \| f \|_{\infty}}{\cntm{F_n}},
\]
and the statement of the lemma follows since $(F_n)_{n \geq 1}$ is a
right F{\o}lner sequence. 
\end{proof}

\begin{rem}
\label{r.lsnotinv}
The statement of Lemma \ref{l.lsinv1} does not hold for general
amenable groups and unbounded nonnegative functions. As a
counterexample, take $\Gamma:=\Z$ with the tempered two-sided
F{\o}lner sequence
\[
F_n:=[-2^n,\dots,2^n] \quad \text{ for } n \geq 1
\]
and define $f: \Gamma \to \N$ to be zero everywhere, except for points
of the form $2^k+1$, where we let
\[
f(2^k+1):=2^k \quad \text{ for all } k \geq 0.
\]
It is then easy to see that
\[
\limsup\limits_{n \to \infty} \avg{g \in F_n} f(g) \neq \limsup\limits_{n
  \to \infty} \avg{g \in F_n} f(g + 1).
\]
We will resolve this issue in the class of groups of polynomial
growth in Lemma \ref{l.lsinv2} in Section \ref{ss.cgpg}.
\end{rem}

\subsection{Computable Groups of Polynomial Growth}
\label{ss.cgpg}
Let $\Gamma$ be a finitely generated discrete group and $\{
\gamma_1,\dots,\gamma_k\}$ be a fixed generating set. Each element
$\gamma \in \Gamma$ can be written as a product
$\gamma_{i_1}^{p_1} \gamma_{i_2}^{p_2} \dots \gamma_{i_l}^{p_l}$ for
some indexes $i_1,i_2,\dots,i_l \in \{ 1,\dots,k \}$ and some integers $p_1,p_2,\dots,p_l \in \Z$. We define the \textbf{norm} of an element $\gamma
\in \Gamma$ by
\[
\| \gamma \|:=\inf\{ \sum\limits_{i=1}^l |p_i|: \gamma =
\gamma_{i_1}^{p_1} \gamma_{i_2}^{p_2} \dots \gamma_{i_l}^{p_l} \},
\]
where the infinum is taken over all representations of $\gamma$ as a
product of the generating elements. The norm $\| \cdot \|$ on $\Gamma$
can, in general, depend on the generating
set, but it is easy to show \cite[Corollary 6.4.2]{ceccherini2010} that
two different generating sets produce equivalent norms. We will always
say what generating set is used in the definition of a norm, but we will
omit an explicit reference to the generating set later on. 

We say that the group $\Gamma$ is of
\textbf{polynomial growth} if there are constants $C,d>0$ such that
for all $n \geq 1$ we have
\[
\cntm{\bl(n)} \leq C n^d.
\]
\begin{exa}
\label{ex.zdex}
Consider the group $\Z^d$ for $d \in \N$ and let $\gamma_1,\dots,\gamma_d \in \Z^d$ be the
standard basis elements of $\Z^d$. That is, $\gamma_i$ is defined by
\[
\gamma_i(j):=\delta_i^j \quad (j=1,\dots, d)
\] 
for all $i=1,\dots,d$. We consider the generating set given by elements $\sum\limits_{k \in I} (-1)^{\varepsilon_k}\gamma_k$ for all
subsets $I \subseteq [1,d]$ and all functions $\varepsilon_{\cdot} \in
\{ 0,1\}^I$. Then it is easy
to see by induction on dimension that $\bl(n) = [-n,\dots,n]^d$, hence
\[
\cntm{\bl(n)} = (2n+1)^d \quad \text{ for all } n \in \N
\]
with respect to this generating set, i.e., $\Z^d$ is a group of polynomial growth.
\end{exa}

Let $d \in \Zp$. We say that the group $\Gamma$ has \textbf{polynomial growth
of degree $d$} if there is a constant $C>0$ such that
\[
\frac 1 C n^d \leq \cntm{\bl(n)} \leq C n^d \quad \text{ for all } n \in \N.
\]
It was shown in \cite{bass1972} that, if $\Gamma$ is a finitely
generated nilpotent group, then $\Gamma$ has polynomial growth of some
degree $d \in \Zp$. Furthermore, one can show \cite[Proposition
6.6.6]{ceccherini2010} that if $\Gamma$ is a group and $\Gamma' \leq \Gamma$
is a finite index, finitely generated nilpotent subgroup, having
polynomial growth of degree $d \in \Zp$, then the group $\Gamma$ has
polynomial growth of degree $d$. The
converse is true as well: it was proved in \cite{gromov1981}
that, if $\Gamma$ is a group of polynomial growth, then there exists a
finite index, finitely generated nilpotent subgroup $\Gamma' \leq
\Gamma$. It follows that if $\Gamma$ is a group of polynomial growth
with the growth function $\gamma$, then there is a constant $C>0$ and
an integer $d\in \Zp$, called the \textbf{degree of polynomial growth}, such that
\[
\frac 1 C n^d \leq \cntm{\bl(n)} \leq C n^d \quad \text{ for all } n \in \N.
\]
An even stronger result was obtained in \cite{pansu1983}, where it is
shown that, if $\Gamma$ is a group of polynomial growth of degree $d
\in \Zp$, then the limit
\begin{equation}
\label{eq.pansu}
c_{\Gamma}:=\lim\limits_{n \to \infty} \frac{\cntm{\bl(n)}}{n^d}
\end{equation}
exists. 

\begin{lemma}
\label{l.gpgamen}
Let $\Gamma$ be a group of polynomial growth. Then $(\bl(n))_{n \geq
1}$ is a tempered two-sided F{\o}lner sequence in $\Gamma$.
\end{lemma}
\begin{proof}
We want to show that for every $g \in \Gamma$
\[
\lim\limits_{n \to \infty} \frac{\cntm{g \bl(n) \sdif
    \bl(n)}}{\cntm{\bl(n)}} = 0.
\]
Let $m:=\| g \| \in \Zp$. Then $g \bl(n) \subseteq \bl(n+m)$, hence
\[
\frac{\cntm{g \bl(n) \sdif
    \bl(n)}}{\cntm{\bl(n)}} \leq \frac{\cntm{\bl(n+m)} - \cntm{ \bl(n)}} {\cntm{\bl(n)}} \to 0,
\]
where we use the existence of the limit in Equation
\eqref{eq.pansu}. Similarly, we use the relation 
$\bl(n) g \subseteq \bl(n+m)$ to show that $(\bl(n))_{n \geq 1}$ is a
right F{\o}lner sequence. The sequence $(\bl(n))_{n \geq 1}$ is
tempered, since
\[
\cntm{\bl(n-1) \cdot \bl(n)} \leq \cntm{\bl(2n)} \leq C^2 \cntm{\bl(n)}
\]
for all $n \geq 1$.
\end{proof}

As promised in Remark \ref{r.lsnotinv}, we prove now that the
$\limsup$ of averages of \emph{arbitrary} nonnegative functions on a
group of polynomial growth $\Gamma$ is
translation invariant.
\begin{lemma}[Limsup invariance]
\label{l.lsinv2}
Let $\Gamma$ be a group of polynomial growth and define the F{\o}lner
sequence of balls around $\ue \in \Gamma$ by
\[
F_n:= \{ g \in \Gamma: \| g \| \leq n\} \quad \text{for } n \geq 1.
\]
Let $f: \Gamma \to \Rp$ be a
nonnegative function on $\Gamma$. Then
\[
\limsup\limits_{n \to \infty} \avg{g \in F_n} f(g) = \limsup\limits_{n
  \to \infty} \avg{g \in F_n} f(g h)
\]
for all $h \in \Gamma$.
\end{lemma}
\begin{proof}
Let $S \subset \Gamma$ be the finite generating set, which is used in
the definition of the norm $\| \cdot \|$ on $\Gamma$. Since the
statement of the lemma is `symmetric' and since the set $S$ generates $\Gamma$, it suffices to prove that
\[
\limsup\limits_{n \to \infty} \avg{g \in F_n} f(g) \geq \limsup\limits_{n
  \to \infty} \avg{g \in F_n} f(g h)
\]
for all $h \in S \cup S^{-1}$. We fix an element $h \in S \cup
S^{-1}$. It is clear that $F_n h \subseteq F_{n+1}$, hence
\[
\limsup\limits_{n
  \to \infty} \avg{g \in F_n} f(g h) \leq \limsup\limits_{n \to
  \infty} \frac{1}{\cntm{F_n}} \sum\limits_{g \in F_{n+1}} f(g).
\]
But
\[
\limsup\limits_{n \to
  \infty} \frac{1}{\cntm{F_n}} \sum\limits_{g \in F_{n+1}} f(g) = \limsup\limits_{n \to
  \infty} \frac{\cntm{F_{n+1}}}{\cntm{F_n}} \cdot \avg{g \in F_{n+1}} f(g)
\]
and $\cntm{F_{n+1}}/\cntm{F_n} \to 1$ as $n \to \infty$, which implies that
\[
\limsup\limits_{n \to
  \infty} \frac{1}{\cntm{F_n}} \sum\limits_{g \in F_{n+1}} f(g) = \limsup\limits_{n \to \infty} \avg{g \in F_n} f(g),
\]
and the proof is complete.
\end{proof}

A computable group $\Gamma$ with a distinguished set of
generators $\{ \gamma_1,\dots, \gamma_k\}$ will be called a \textbf{computable
group of polynomial growth} if $\Gamma$ is a group of polynomial
growth. It will be essential further that the generating set is \emph{known
and fixed}. More precisely, we state the following lemma.
\begin{lemma}
Let $\Gamma$ be a computable group of polynomial growth
with a distinguished set of generators $\{ \gamma_1,\dots,
\gamma_k\}$. Then the following assertions hold:
\begin{aufzi}
\item The sequence of balls $(\bl(n))_{n \geq 1}$ is a canonically
  computable sequence of finite sets;
\item The growth function $n \mapsto \cntm{\bl(n)}, \Zp \to \N$ is a total computable function;
\item The norm $\| \cdot \|: \Gamma \to \Zp$ is a total computable function.
\end{aufzi}
\end{lemma}
\noindent The proof of the lemma is straightforward.

\subsection{Ergodic Theory}
\label{ss.et}
Let $\prX=(X,\calB,\mu)$ be a probability space. A measurable transformation
$\varphi: X \to X$ is called \textbf{measure-preserving} if
\[
\mu(\varphi^{-1} A) = \mu(A) \quad \text{ for all } A \in \calB.
\]
A measure-preserving transformation $\varphi: X \to X$ is called an
\textbf{automorphism} if there exists a measure-preserving
transformation $\psi: X \to X$ such that
\[
\varphi \circ \psi = \psi \circ \varphi = \mathrm{id}_X \quad \mu-\text{a.e.}
\]
We denote by $\Aut(\prX)$ the group of all automorphisms of the
probability space $\prX$. Given a discrete group $\Gamma$, a
\textbf{measure-preserving $\Gamma$-system}\footnote{To simplify the
  notation, the shorter term `$\Gamma$-system' will also be used.} is a probability space $\prX=(X,\calB,\mu)$, endowed with
an action of $\Gamma$ on $X$ by automorphisms from $\Aut(\prX)$. We denote a
measure-preserving $\Gamma$-system on a probability space
$(X,\calB,\mu)$ by a triple $(X,\mu,\Gamma)$ and we write $g \cdot x$,
where $g \in \Gamma, x\in X$, to denote the corresponding action of
$\Gamma$ on elements of $X$.

Let $\bfX = (X,\mu, \Gamma)$ be a measure-preserving
$\Gamma$-system on a probability space $(X,\calB,\mu)$. We say that $\bfX$ is \textbf{ergodic} (or that the
measure $\mu$ on $X$ is ergodic) if, for all $A \in \calB$, the
condition 
\[
\mu(\gamma^{-1} A \sdif A) = 0\quad \text{ for all } \gamma \in \Gamma
\]
implies that $\mu(A)=0$ or $\mu(A)=1$. That is, $\bfX$ is ergodic if only the
trivial sets are essentially invariant under $\Gamma$. 

The simplest ergodic theorem for amenable group actions is the mean
ergodic theorem, which we state below. For the proof we refer the reader to \cite[Theorem 3.33]{glasner}.

\begin{thm}
\label{t.meanerg}
Let $(X, \mu ,\Gamma)$ be a measure-preserving, ergodic $\Gamma$-system,
where the group $\Gamma$ is amenable and $(F_n)_{n \geq 1}$ is a left
F{\o}lner sequence. Then for every $f \in \Ell{2}(X)$ we have
\begin{equation*}
\lim\limits_{n \to \infty} \avg{g \in F_n} f \circ g = \int f d\mu,
\end{equation*}
where the convergence is understood in $\Ell{2}(X)$-sense. 
\end{thm}

Pointwise convergence of ergodic averages is much more tricky, in
particular, pointwise ergodic averages do not necessarily converge,
unless the F{\o}lner sequence satisfies some additional assumptions. The
following important theorem was proved by E. Lindenstrauss in
\cite{lindenstrauss2001}\footnote{In fact, a more general statement is proved there, but we only need the ergodic case in this work.}.
\begin{thm}
\label{t.pointwiseet}
Let $\bfX=(\prX,\mu,\Gamma)$ be an ergodic measure-preserving $\Gamma$-system, where the group $\Gamma$ is amenable and $(F_n)_{n \geq 1}$ is a tempered left F{\o}lner sequence. Then for every $f \in \Ell{1}{(X)}$
\begin{equation*}
\lim\limits_{n \to \infty} \avg{g \in F_n} f(g \cdot x) = \int f d\mu
\end{equation*}
for $\mu$-a.e. $x \in X$.
\end{thm}

\subsection{Computability on Cantor Space and Martin-L\"{o}f Randomness}
\label{ss.ccs}

In this section we remind the reader some standard notions of
computability on Cantor space. All of these notions have analogs on
computable metric spaces as well, and we refer to
\cite{hoyrup2009}, \cite{galatolo2010} for the details.

Throughout the article we fix some enumeration of $\Q = \{
q_1,q_2,q_2,\dots\}$. We use the standard notions of a computable real
number and of a lower/upper
semi-computable real number. A sequence of real numbers $(a_n)_{n \geq 1}$ is
called \textbf{computable uniformly in $n$} if there exists an
algorithm $A: \N \times \N \to \Q$ such that
\[
|A(n,i)-a_n|<2^{-i} \quad \text{ for all } n,i \geq 1.
\]

We fix some enumeration 
\[
\{ 0,1\}^{\ast} = \{ \uw_1,\uw_2,\dots\}
\]
of the set if all finite binary words, and let
\[
[\uw]:=\{ \omega: \omega = \uw \omega' \text{ for some } \omega' \in \{ 0,1\}^{\N}\} \subset \{ 0,1\}^{\N}
\]
be the cylinder set of all words that begin with a finite word $\uw \in \{ 0,1\}^{\ast}$.
A set $U \subseteq \{ 0,1\}^{\N}$ is called \textbf{effectively open} if there is a recursively enumerable subset $E \subseteq \N$ such that
\[
U = \bigcup\limits_{i \in E} [\uw_i].
\]
A sequence $(U_n)_{n \geq 1}$ of sets is called a \textbf{uniformly effectively open} sequence of sets if there is a recursively enumerable set $E \subseteq \N \times \N$ such that
\[
U_i = \bigcup\limits_{(i,j) \in E} [\uw_j] \quad \text{ for all } i \geq 1.
\]

A mapping $\varphi: \{ 0,1\}^{\N} \to \{ 0,1\}^{\N}$ is called \textbf{computable} if $(\varphi^{-1}([\uw_i]))_{i\geq 1}$ is uniformly effectively open, that is, there is a recursively enumerable set $E_{\varphi} \subseteq \N\times \N$ such that
\[
\varphi^{-1}([\uw_i]) = \bigcup\limits_{(i,j) \in E_{\varphi}} [\uw_j] \quad \text{ for all } i \geq 1.
\]
A function $f: \{ 0,1\}^{\N} \to \Rp$ is called \textbf{lower
  semicomputable} if the sequence of sets $(f^{-1}((q_n,+\infty)))_{n
  \geq 1}$ is uniformly effectively open.

Let $\mu$ be a Borel probability measure on $\{ 0,1\}^{\N}$. We say that $\mu$ is a \textbf{computable measure} if
\[
\mu([\uw_{i_1}] \cup [\uw_{i_2}] \cup \dots \cup [\uw_{i_k}])
\]
is computable uniformly in $i_1,\dots,i_k \geq 1$.

Suppose that $\mu$ is a computable probability measure on $\{ 0,1\}^{\N}$. A \textbf{Martin-L\"{o}f $\mu$-test} is a uniformly effectively open sequence of sets $(U_n)_{n \geq 1}$ such that
\[
\mu(U_n) < 2^{-n} \quad \text{ for all } n \geq 1.
\]
Any subset of $\bigcap\limits_{n \geq 1} U_n$ is called an \textbf{effectively $\mu$-null set}. A point $\omega \in \{ 0,1\}^{\N}$ is called \textbf{Martin-L\"{o}f random} if it is not contained in any effectively $\mu$-null set.

\subsection{Computable Dynamical Systems}
\label{ss.cps}

Now, let $\Gamma \subseteq \N$ be a computable group, which acts on $\{ 0,1\}^{\N}$ by homeomorphisms. We say that the action of $\Gamma$ is \textbf{computable} if there is a recursively enumerable subset $E \subseteq \Gamma \times \N \times \N$ such that
\[
\gamma^{-1}([\uw_i]) = \bigcup\limits_{(\gamma,i,j) \in E} [\uw_j] \quad \text{ for all } i \geq 1, \gamma \in \Gamma.
\]

In general, checking the computability of the action of a computable
group $\Gamma$ on $\{ 0,1 \}^{\N}$ can be trickier than checking
computability of a single transformation. Imagine a $\Z$-action on $\{
0,1\}^{\N}$ with the generating element $\varphi \in \Z$. Can it
happen that both $\varphi$ and $\varphi^{-1}$ are computable
transformations of $\{ 0,1\}^{\N}$, whilst the action of $\Z$ on $\{
0,1\}^{\N}$ is not computable? Fortunately, the answer is `no': the
following lemma tells us that for an action of a computable finitely generated group it suffices to check computability of transformations in a finite symmetric generating set to guarantee the computability of the action. The lemma also shows that the terminology of computable group actions which we suggest in this article is compatible with the classical case, when there is only one computable transformation.

\begin{lemma}
\label{l.fgactcomp}
Let $\Gamma$ be a finitely generated computable group with a finite symmetric generating set $S \subset \Gamma$. Suppose that $\Gamma$ acts on $\{ 0,1\}^{\N}$ by homeomorphisms, and, furthermore, that for each $\gamma \in S$ the transformation
\[
\gamma: \{ 0,1 \}^{\N} \to \{ 0,1 \}^{\N}
\]
is computable. Then the action of $\Gamma$ on $\{ 0,1\}^{\N}$ is computable.
\end{lemma}
\begin{proof}
Given a fixed finite symmetric generating set $S = \{ \gamma_1, \gamma_2, \dots, \gamma_N\}$, we will denote by $\bl(n)$ the corresponding balls around the neutral element $\ue \in \Gamma$ with respect to the norm determined by $S$. Since $\gamma_1,\gamma_2,\dots,\gamma_N$ are computable endomorphisms of $\{ 0,1\}^{\N}$, there are recursively enumerable subsets $E_1,E_2,\dots,E_N$ such that
\[
\gamma_k^{-1}([\uw_i]) = \bigcup\limits_{(i,j) \in E_k} [\uw_j] \quad \text{ for all } k=1,\dots,N, i \geq 1.
\]

We will describe an algorithm, which enumerates the set $E$. At \emph{stage} $n$, the algorithm first computes the finite set $\bl(n) \subset \Gamma$ by computing all products of the elements of $S$ of length at most $n$. For each word
\[
\gamma_{i_1} \gamma_{i_2} \dots \gamma_{i_k} = \gamma \in \bl(n)
\]
we have for all $i \geq 1$ 
\begin{align*}
\gamma^{-1}([\uw_i]) &= \gamma_{i_k}^{-1} \gamma_{i_{k-1}}^{-1}\dots \gamma_{i_1}^{-1}([\uw_i]) = \gamma_{i_k}^{-1} \gamma_{i_{k-1}}^{-1}\dots \gamma_{i_2}^{-1} \bigcup\limits_{(i,j_1) \in E_{i_1}} [\uw_{j_1}] = \\
&=\gamma_{i_k}^{-1} \gamma_{i_{k-1}}^{-1}\dots \gamma_{i_3}^{-1} \bigcup\limits_{(i,j_1) \in E_{i_1}} \bigcup\limits_{(j_1,j_2) \in E_{i_2}} [\uw_{j_2}] = \\
&=\bigcup\limits_{(i,j_1) \in E_{i_1}} \bigcup\limits_{(j_1,j_2) \in E_{i_2}} \bigcup\limits_{(j_2,j_3) \in E_{i_3}} \dots \bigcup\limits_{(j_{k-1},j_k) \in E_{i_k}}  [\uw_{j_k}].
\end{align*}
We compute the first $n$ pairs $(i,j_1) \in E_{i_1}$, for each of these pairs we compute the first $n$ pairs $(j_1,j_2) \in E_{i_2}$ and so on up to the first $n$ pairs $(j_{k-1},j_k) \in E_{i_k}$ (where $j_{k-1}$ comes from the one but the last step). The algorithm prints all resulting triples $(\gamma, i, j_k)$, and proceeds to the next word (or the next stage, if all words at the current stage have been exhausted).

Since, at each stage $n$, we look through \emph{all} products of length at most $n$, it is easy to see that
\[
\gamma^{-1}([\uw_i]) = \bigcup\limits_{(\gamma,i,j) \in E} [\uw_j]
\]
for all $i \geq 1$, and, furthermore, the set $E$ is recursively enumerable.
\end{proof}

A \textbf{computable Cantor $\Gamma$-system}\footnote{Or a
  \emph{computable $\Gamma$-system} for short, since we only consider
  dynamical systems on Cantor space in this article.} is a triple $(\{ 0,1\}^{\N},\mu,\Gamma)$,
where $\mu$ is a computable measure on $\{ 0,1\}^{\N}$ and $\Gamma$ acts computably on $\{ 0,1\}^{\N}$ by measure-preserving transformations.

\begin{rem}
\label{r.compact}
The notion of a computable action of a computable group which we
suggest directly translates to arbitrary computable metric
spaces. Furthermore, Lemma \ref{l.fgactcomp} remains valid in the
more general setting. 
\end{rem}

\section{Effective Birkhoff's Theorem}
\subsection{Ku\v{c}era's Theorem}
\label{ss.kt}

In this section we generalize Ku\v{c}era's theorem for computable
actions of amenable groups. In the proof we follow roughly the approach from
\cite{shen2012}, although the technical details do differ.
\begin{thm}
\label{t.kt}
Let $\Gamma$ be a computable amenable group and $(\{ 0,1\}^{\N},\mu,\Gamma)$ be a
computable ergodic $\Gamma$-system. Let $U \subset X$ be an effectively
open subset  such that $\mu(U)<1$. Let 
\[
U^{\ast}:=\bigcap\limits_{g \in \Gamma} g^{-1} U
\]
be the set of all points $\omega \in \{0,1\}^{\N}$ whose orbit remains in
$U$. Then $U^{\ast}$ is an effectively null set.
\end{thm}

\begin{proof}
Let $(F_n)_{n \geq 1}$ be a canonically computable two-sided F{\o}lner
sequence in $\Gamma$ and $\mu(U)< q < 1$ be some fixed rational
number. Let $(I_i)_{i \geq 1}$ be the basis of cylinder sets in $(\{0,1\}^{\N},\mu)$. Let
\begin{equation}
\label{eq.ini}
(i,k) \mapsto n(i,k), \quad \N \times \N \to \N
\end{equation}
be some total computable function, which will be chosen later, and
define a computable function $m$ by
\[
m(i,k):=\cntm{F_{n(i,k)}} \quad \text{ for } i,k \geq 1.
\]

Since $U_0:=U$ is effectively open, there is a r.e. subset $E_0 \subseteq \N$ such that $U_0 =
\bigcup\limits_{i \in E_0} I_i$ is a union of disjoint cylinder sets. Since the action of $\Gamma$ is
computable and since $(F_n)_{n \geq 1}$ is canonically computable, the sequence 
\[
i \mapsto I_i \cap \bigcap\limits_{g \in F_{n(i,1)}} g^{-1}(U_0)
\]
is a uniformly effectively open sequence of sets. Let 
\[
U_1:=\bigcup\limits_{i \in E_0} \left( I_i \cap
  \bigcap\limits_{g \in F_{n(i,1)}} g^{-1}(U_0) \right),
\]
then, clearly, $U_1 \subseteq U_0$ is an effectively open set and $U^{\ast}
\subseteq U_1$. Since $U_1$ is an effectively open set, there is a r.e. subset
$E_1 \subseteq \N$ such that $U_1 = \bigcup\limits_{i \in E_1}
I_i$ is a union of disjoint cylinder sets. Suppose that
\begin{equation}
\label{eq.expdec1}
\mu \left( I_i \cap \bigcap\limits_{g \in F_{n(i,1)}} g^{-1}(U)
\right) < q \mu(I_i) + q \cdot 2^{-i} \quad \text{ for all } i \geq 1.
\end{equation}
The cylinder sets $(I_i)_{i \in E_0}$ are pairwise disjoint, hence $\mu(U_1) \leq q \mu(U_0)+q$. 

We want to apply the same procedure
to $U_1$ and so on to obtain a sequence of uniformly open sets with
almost exponentially decaying measure. So, in general, let $k \geq 1$ and suppose that $U_{k-1} =
\bigcup\limits_{i \in E_{k-1}} I_i$ is a disjoint union of cylinder sets for an r.e. subset
$E_{k-1}$. We let
\begin{align*}
U_k:=\bigcup\limits_{i \in E_{k-1}} &\left( I_i \cap
  \bigcap\limits_{g \in F_{n(i,k)}} g^{-1}(U_{k-1}) \right) = \bigcup\limits_{i \in E_{k-1}} \left( I_i \cap
  \bigcap\limits_{g \in F_{n(i,k)}} \bigcup\limits_{j \in E_{k-1}}
  g^{-1}(I_j) \right)=\\
&=\bigcup\limits_{i \in E_{k-1}}  
  \bigcup\limits_{j_1,\dots, j_{m(i,k)} \in E_{k-1}} \left( I_i \cap \bigcap\limits_{s=1}^{m(i,k)} 
  g_{i,s}^{-1}(I_{j_s}) \right),
\end{align*}
where $g_{i,1},g_{i,2},\dots,g_{i,m(i,k)}$ is the list of all distinct
elements of $F_{n(i,k)}$. The sequence of sets 
\[
\left( I_i
\cap \bigcap\limits_{s=1}^{m(i,k)} g_{i,s}^{-1}(I_{j_s})
\right)_{i,j_1,\dots,j_{m(i,k)}}
\]
is uniformly effectively open, so it follows that $(U_k)_{k \geq 1}$
is uniformly effectively open. Clearly, $U^{\ast}
\subseteq U_k \subseteq U_{k-1}$ for every $k \geq 1$. If we show that
\begin{equation}
\label{eq.expdeck}
\mu \left( I_i \cap \bigcap\limits_{g \in F_{n(i,k)}} g^{-1}(U_{k-1})
\right) < q \mu(I_i) + q^k \cdot 2^{-i} \quad \text{ for all } i \geq 1 
\end{equation}
then $\mu(U_k) < q \mu(U_{k-1})+q^k$ for every $k$, and so $\mu(U_k) <
(k+1) q^k$, which would imply that $U^{\ast}$ is an effectively null set. Observe that
\begin{align*}
&\mu \left( I_i \cap \bigcap\limits_{g \in F_{n(i,k)}} g^{-1}(U_{k-1})
\right) \leq \min\limits_{g \in F_{n(i,k)}} \mu \left( I_i \cap  g^{-1}(U_{k-1})
\right) \leq \\
& \leq \avg{g \in F_{n(i,k)}} \mu \left( I_i \cap g^{-1}(U_{k-1})
\right) = \avg{g \in F_{n(i,k)}} \mu \left( (g I_i) \cap U_{k-1}
\right) = \\
&=\int \left( \avg{g \in F_{n(i,k)}^{-1}} \indi{I_i}(g \cdot \omega) \right) \indi{U_{k-1}}(\omega) d \mu.
\end{align*}
If, for every $i,k \geq 1$, we find effectively a number $n(i,k)$ such that
\begin{equation}
\label{eq.l2normest}
\| \avg{g \in F_{n(i,k)}^{-1}} \indi{I_i}(g \cdot \omega) -
\mu(I_i)\|_2 < q^k \cdot 2^{-i},
\end{equation}
then, due to Cauchy-Schwarz inequality, the computation above implies that
\[
\mu \left( I_i \cap \bigcap\limits_{g \in F_{n(i,k)}} g^{-1}(U_{k-1}) 
\right) \leq q \mu(I_i) +q^k \cdot 2^{-i}.
\]
Mean ergodic theorem (Theorem \ref{t.meanerg}) implies that a number $n(i,k)$
satisfying Equation \eqref{eq.l2normest} always exists, since $(
F_n^{-1} )_{n\geq 1}$ is a left F{\o}lner sequence. To find the number $n(i,k)$
\emph{effectively} we argue as follows. 

First, $(g I_i)_{g \in \Gamma, i \geq 1}$ is a uniformly effectively
open sequence of sets by definition of computability of the action of
$\Gamma$ on $\{0,1\}^{\N}$, so let $E \subseteq \Gamma \times \N \times \N$ be an
r.e. subset such that
\[
g (I_i) = \bigcup\limits_{(g,i,j) \in E} I_j \quad \text{ for all
} g \in \Gamma, i \geq 1
\]
We claim that there exists a uniformly effectively
open sequence of sets $(\Delta_{g,i}^k)_{g,i,k}$, where each
$\Delta_{g,i}^k$ is the union of the first $\cntm{\Delta_{g,i}^k}$
intervals in $g I_i$, such that the function $(g,i,k) \mapsto
\cntm{\Delta_{g,i}^k}$ is total computable and that
\begin{equation}
\label{eq.deltapp}
\mu(g I_i \setminus \Delta_{g,i}^k )< \frac{q^{2k} \cdot 2^{-2i}}{64}
\quad \text{ for all } g \in \Gamma \text{ and } i,k \geq 1.
\end{equation}
To do so, we use computability of the measure $\mu$ to find (uniformly in $i,k$ and effectively) a rational $d_i^k$ such that
\[
|\mu(I_i) - d_i^k|<\frac{q^{2k} \cdot 2^{-2i}}{256} \quad \text{ for
  all } i,k \geq 1.
\]
The set $\Delta_{g,i}^k$ is constructed as follows. Let
$\Delta_{g,i}^k=\varnothing$. Take the first
interval $I_{j_1}$ such that $(g,i,j_1) \in E$, add it to the
collection $\Delta_{g,i}^k$ and compute its measure
$\widetilde m_{g,i}$ with precision $\frac{q^{2k} \cdot
  2^{-2i}}{256}$. If
\begin{equation}
\label{eq.dgikconst}
\widetilde m_{g,i}> d_i^k-\frac{q^{2k} \cdot
  2^{-2i}}{128},
\end{equation}
then we are done. Otherwise, we add the next interval $I_{j_2}$ such
that $(g,i,j_2) \in E$ to the collection $\Delta_{g,i}^k$, compute the 
measure $\widetilde m_{g,i}$ of the union of intervals in $\Delta_{g,i}^k$ with precision $\frac{q^{2k} \cdot
  2^{-2i}}{256}$ and check the condition \eqref{eq.dgikconst} once
again and so on. The algorithm eventually terminates, it is clear that it
provides a uniformly effectively open sequence of sets
$(\Delta_{g,i}^k)_{g,i,k}$, and a direct computation shows that
condition \eqref{eq.deltapp} is satisfied as well.

The number $n(i,k)$ is defined as the smallest nonnegative integer
such that
\[
\| \avg{g \in F_{n(i,k)}} \indi{\Delta_{g,i}^k} - d_i^k\|_2 < \frac{q^k
\cdot 2^{-i}}{2},
\]
where the $\Ell{2}$-norm is computed, say, with a $\frac{q^{2k} \cdot
  2^{-2i}}{256}$-precision. Such $n(i,k)$ exists due to Mean Ergodic
Theorem and our choice of the sets $\Delta_{g,i}^k$. Furthermore, it
is computable, since the sequence of sets $(\Delta_{g,i}^k)$ is
uniformly effectively open, the measure $\mu$ is computable and
$(F_{n})_{n \geq 1}$ is a computable F{\o}lner sequence.
\end{proof}

\subsection{Birkhoff's Theorem}
\label{ss.mt}

In this section we prove the main theorems of the article. Our main
technical tools are the generalization of Ku\v{c}era's theorem from
the previous section, the
result of Lindenstrauss about pointwise convergence of ergodic
averages and Lemmas \ref{l.lsinv1}, \ref{l.lsinv2} about the
invariance of limsup of averages. First, we prove Birkhoff's effective
ergodic theorem for indicator functions of effectively opens sets.
\begin{lemma}
\label{l.ebirk2stamen}
Let $\Gamma$  be a computable amenable group with a canonically computable tempered
two-sided F{\o}lner sequence $(F_n)_{n \geq 1}$. Suppose that
$(\{0,1\}^{\N},\mu,\Gamma)$ is a computable ergodic Cantor system and that $U
\subseteq \{0,1\}^{\N}$ is an effectively open set. For every Martin-L\"{o}f random $\omega \in \{0,1\}^{\N}$
the equality
\[
\lim\limits_{n \to \infty} \avg{g \in F_n} \indi{U}(g \cdot \omega) = \mu(U)
\]
holds.
\end{lemma}
\begin{proof}
First, let us show that
\[
\limsup \limits_{n \to \infty} \avg{g \in F_n} \indi{U}(g \cdot
\omega) \leq \mu(U)
\]
for every Martin-L\"{o}f random $\omega$. Let 
\[
q > \mu(U)
\] 
be some fixed rational number. Let
\[
A_k:=\{ x \in \{0,1\}^{\N}: \sup\limits_{n \geq k} \avg{g \in F_n} \indi{U}(g
\cdot x) > q\} \quad \text{ for } k \geq 1,
\]
which is an effectively open set. Pointwise ergodic theorem (Theorem \ref{t.pointwiseet}) implies that $\mu(\bigcap\limits_{k \geq 1} A_k) = 0$, hence there
is some $k \geq 1$ such that $\mu(A_k)<1$. Let $\omega \in \{0,1\}^{\N}$ be an arbitrary
Martin-L\"{o}f random point. It follows from Theorem
\ref{t.kt} that $\omega \notin A_k^{\ast}$, hence there exists $g_0 \in \Gamma$ such that $g_0 \cdot \omega
\notin A_k$. Hence
\[
\limsup\limits_{n \geq 1} \avg{g \in F_n} \indi{U}(g \cdot (g_0 \cdot \omega)) \leq q.
\]
The function $g \mapsto \indi{U}(g \cdot \omega)$ on $\Gamma$ is
bounded, thus we can use Lemma \ref{l.lsinv1} to deduce that
\[
\limsup\limits_{n \geq 1} \avg{g \in F_n} \indi{U}(g \cdot \omega) = \limsup\limits_{n \geq 1} \avg{g \in F_n} \indi{U}(g \cdot (g_0 \cdot \omega)) \leq q.
\]
Since $q > \mu(U)$ is an arbitrary rational, this implies that $
\limsup\limits_{n \geq 1} \avg{g \in F_n} \indi{U}(g \cdot \omega)
\leq \mu(U)$.

Secondly, if $U = \bigcup\limits_{i \in E} I_i$ for an r.e. subset
$E \subseteq \N$, we let $\Delta_k \subseteq U$ be the union $I_{i_1},\dots,I_{i_k}$ of the first $k$
intervals in $U$ for every $k \geq 1$. Then $\Delta_k$ is a clopen
subset, and its compliment $\Delta_k^c$ is an effectively open
set. The preceding argument, applied to $\Delta_k^c$, implies that
\[
\limsup\limits_{n \geq 1} \avg{g \in F_n} \indi{\Delta_k^c}(g \cdot \omega)
\leq \mu(\Delta_k^c)=1-\mu(\bigcup\limits_{j=1}^k I_{i_j}).
\]
Since $k \geq 1$ is arbitrary, it follows easily that
\[
\mu(U) \leq \liminf\limits_{n \geq 1} \avg{g \in F_n} \indi{U}(g \cdot
\omega)
\]
and the proof is complete.
\end{proof}

We proceed to the main theorems of the article.
\begin{thm}
\label{t.ebirk2amen}
Let $\Gamma$  be a computable amenable group with a canonically computable tempered
two-sided F{\o}lner sequence $(F_n)_{n \geq 1}$. Suppose that
$(\{0,1\}^{\N},\mu,\Gamma)$ is a computable ergodic $\Gamma$-system. For every bounded
lower semicomputable $f$ and for every Martin-L\"{o}f random $\omega \in \{0,1\}^{\N}$
the equality
\[
\lim\limits_{n \to \infty} \avg{g \in F_n} f(g
\cdot \omega) = \int\limits f d \mu
\]
holds.
\end{thm}
\begin{proof}
Firstly, the proof that
\[
\limsup \limits_{n \to \infty} \avg{g \in F_n} f(g \cdot
\omega) \leq \int f d \mu
\]
for every Martin-L\"{o}f random $\omega$ is completely analogous to
the first part of the proof of Lemma \ref{l.ebirk2stamen} above. In
particular, the argument about the
translation-invariance of 
\[
\limsup\limits_{n \geq 1} \avg{g \in F_n}
f(g \cdot \omega)
\]
remains valid, since $f$ is a bounded function and we
can once again use Lemma \ref{l.lsinv1}.

Secondly, given an arbitrary $\varepsilon>0$, let $0 \leq h \leq f$ be
a finite linear combination of indicator functions of effectively open sets
such that
\[
\| f - h \|_1 \leq \varepsilon.
\]
An application of Lemma \ref{l.ebirk2stamen} yields that
\[
\liminf\limits_{n \geq 1}\avg{g \in F_n} f(g \cdot \omega) \geq
\liminf\limits_{n \geq 1}\avg{g \in F_n} h(g \cdot \omega) \geq \int h
d\mu \geq \int f d\mu - \varepsilon,
\]
which completes the proof, since $\varepsilon>0$ is arbitrary.
\end{proof}

\begin{rem}
\label{r.stdbirk}
Compared to \cite{shen2012}, we make an additional assumption in
Theorem \ref{t.ebirk2amen} that the observable is \emph{bounded}. The reason for
that is that the invariance of $\limsup$ is only in general guaranteed
by Lemma \ref{l.lsinv1} for bounded functions.
\end{rem}

In a special case, when  $\Gamma$ is a computable group of polynomial
growth, we can remove the additional assumption about the boundedness
of $f$. The theorem below is a generalization of \cite[Theorem 8]{shen2012}.
\begin{thm}
\label{t.ebirk2gpg}
Let $\Gamma$  be a computable group of polynomial growth with the F{\o}lner
sequence of balls around $\ue \in \Gamma$ given by
\[
F_n:= \{ g \in \Gamma: \| g \| \leq n\} \quad \text{for } n \geq 1.
\]
Suppose that $(\{0,1\}^{\N},\mu,\Gamma)$ is a computable ergodic $\Gamma$-system. For every lower semicomputable $f$ and for every Martin-L\"{o}f random $\omega \in \{0,1\}^{\N}$
the equality
\[
\lim\limits_{n \to \infty} \avg{g \in F_n} f(g
\cdot \omega) = \int\limits f d \mu
\]
holds.
\end{thm}
\begin{proof}
The argument is identical to the reasoning in  Theorem \ref{t.ebirk2amen}. We
use Lemma \ref{l.lsinv2} for the invariance of $\limsup$ of averages,
hence obtaining the proof for an \emph{arbitrary} lower semicomputable $f$.
\end{proof}

\printbibliography[]
\end{document}